\documentclass[a4paper,11pt]{article}

\usepackage{amsmath}
\usepackage{amsthm}
\usepackage{amssymb}
\usepackage{amsfonts}
\usepackage{hyperref}
\usepackage[all]{hypcap}[2006/02/20]

\newcommand{\mb}[1]{\mathbf{#1}}
\newcommand{\bb}[1]{\mathbb{#1}}

\newcommand{\pard}[2]{\frac{\partial #1}{\partial #2}}
\newcommand{\mbf}[1]{\mathbf{#1}}
\newcommand{\mbb}[1]{\mathbb{#1}}

\newcommand{\ho}{\left(\frac{d}{dt} -\Delta \right)}
\newcommand{\ddt}[1]{\frac{ d #1}{dt}}
\newcommand{\ip}[2]{\left \langle #1 , #2 \right\rangle}
\newcommand{\n}{\nabla}
\newcommand{\ov}{\overline}
\newcommand{\dmu}{d\check{\mu}}
\newcommand{\mbt}{\boldsymbol{\theta}}
\usepackage[pdftex]{graphicx}

\begin{document}
\theoremstyle{plain}
\newtheorem{theorem}{Theorem}[section]
\newtheorem{lemma}[theorem]{Lemma}
\newtheorem{claim}[theorem]{Claim}
\newtheorem{prop}[theorem]{Proposition}
\newtheorem{cor}[theorem]{Corollary}

\theoremstyle{definition}
\newtheorem{defses}[theorem]{Definition}
\newtheorem{assumption}[theorem]{Assumption}

\theoremstyle{remark}
\newtheorem{remark}[theorem]{Remark}

\begin{center}

{\LARGE \bfseries{The constant angle problem for mean curvature flow inside rotational tori}}\\[20pt]

by \\
\textsc{Ben Lambert} \\
\emph{Bath University}\\
\texttt{b.lambert@bath.ac.uk}
\end{center}

\begin{abstract}
We flow a hypersurface in Euclidean space by mean curvature flow with a Neumann boundary condition, where the boundary manifold is any torus of revolution. If we impose the conditions that the initial manifold is compatible and does not contain the rotational vector field in its tangent space, then mean curvature flow exists for all time and converges to a flat cross-section as $t \rightarrow \infty$. 
\end{abstract}

\begin{center}
\end{center}

\section{Introduction}
We consider Mean Curvature Flow (MCF) of hypersurfaces with a Neumann boundary condition, choosing the boundary manifold to be an $n$-dimensional torus of rotation of any profile embedded in $(n+1)$-dimensional Euclidean space. If the initial manifold is compatible with the boundary condition and is transversal to the rotational vector field, then the flow exists for all time and converges to a flat cross section. A tool regularly used with MCF with a Neumann boundary condition is to assume convexity of the boundary manifold \cite{AltschulerWu}\cite{Stahlsecond}, which gives a sign on boundary derivatives and allows application of Hopf maximum principle. Clearly we cannot impose this in the case of the torus. Instead we observe that we are essentially considering a graphical problem and use a Stampacchia iteration argument similar to that used by Huisken \cite{Huiskengraph} for graphical mean curvature flow. 

Mean curvature flow with a Neumann boundary condition has been considered as graphs in the perpendicular case by Huisken \cite{Huiskengraph}, and also for more general angles in dimension $2$ by Altschuler and Wu \cite{AltschulerWu}. In both cases the flow exists for all time and converges to a special solution: In the first case to a flat plane, while in the second to a translating solution. The level set method in the case of the right angled Neumann condition in a convex cylinder has been studied by Giga, Ohnuma and Sato \cite{GigaOhnumaSatoNeumann}. Considering the perpendicular angle condition further, Stahl\cite{Stahlsecond}\cite{Stahlfirst} showed that if the boundary manifold is a totally umbillic surface and the initial manifold is convex, then under mean curvature flow the manifold shrinks to a point. Furthermore, renormalising homothetically about this point the flow converges to a half sphere. Buckland \cite{Buckland} used a monotonicity argument (again with a perpendicular boundary condition) to classify the Type I boundary singularities of a mean convex initial surface. The perpendicular Neumann boundary condition has also been considered in flat Minkowski space by the author \cite{LambertMinkowski}, where the boundary manifold was chosen to be a convex timelike cone and the rescaled flow converges to a hyperbolic hyperplane. 

Suppose $\Sigma \subset \bb{R}^{n+1}$ is a smooth orientable hypersurface with outward pointing unit normal vector $\mu$. Following Stahl \cite{Stahlfirst} we say \linebreak\mbox{$\mathbf{F}: M^n \times [0,T] \rightarrow \bb{R}^{n+1}$} satisfies \emph{Mean Curvature Flow with a Neumann free boundary condition $\Sigma$} if
\begin{equation}
\label{MCF}
\begin{cases}
\frac{d \mathbf{F}}{dt} = \mathbf{H}= H \nu & \forall (x,t) \in M^n \times [0,T]\\
\mathbf{F}(\cdot,0)= M_0&\\
\mathbf{F}(x,t) \subset \Sigma & \forall (x,t) \in \partial M^n \times [0,T]\\
\ip{\nu}{\mu}(x,t)=0 & \forall (x,t) \in \partial M^n \times [0,T]\ \ ,\\
\end{cases}
\end{equation}
where $\nu(x,t)$ is the normal to $\mbf{F}$ at time $t$. We will often write $M_t=\mb{F}(\cdot, t)$. In this paper we choose $\Sigma$ to be a rotationally symmetric torus of any profile -- topologically $S^1 \times S^{n-1}$  -- and we flow a disk contained within the interior of this torus. At a point $\mb{p}=(p_1, \ldots, p_{n+1}) \in \bb{R}^{n+1}$ we define
\[r=\sqrt{p_n^2 +p_{n+1}^2}, \ \ \ \mb{r} =\frac{1}{r}(0, \ldots, 0, p_n, p_{n+1}), \ \ \ \mbt= \frac{1}{r}(0, \ldots, 0, -p_{n+1}, p_n)\ \ . \]

We prove the following:
\begin{theorem}
 Suppose $\Sigma$ is a torus of rotation and $M_0$ is an embedded disk satisfying the boundary condition which nowhere contains the vector field $\mbt$ in its tangent space, then a solution to equation (\ref{MCF}) with initial data $M_0$ exists for all time and converges uniformly to a flat cross-section of the torus.
\end{theorem}

\begin{figure} \label{jestershat}
\includegraphics[width=\textwidth]{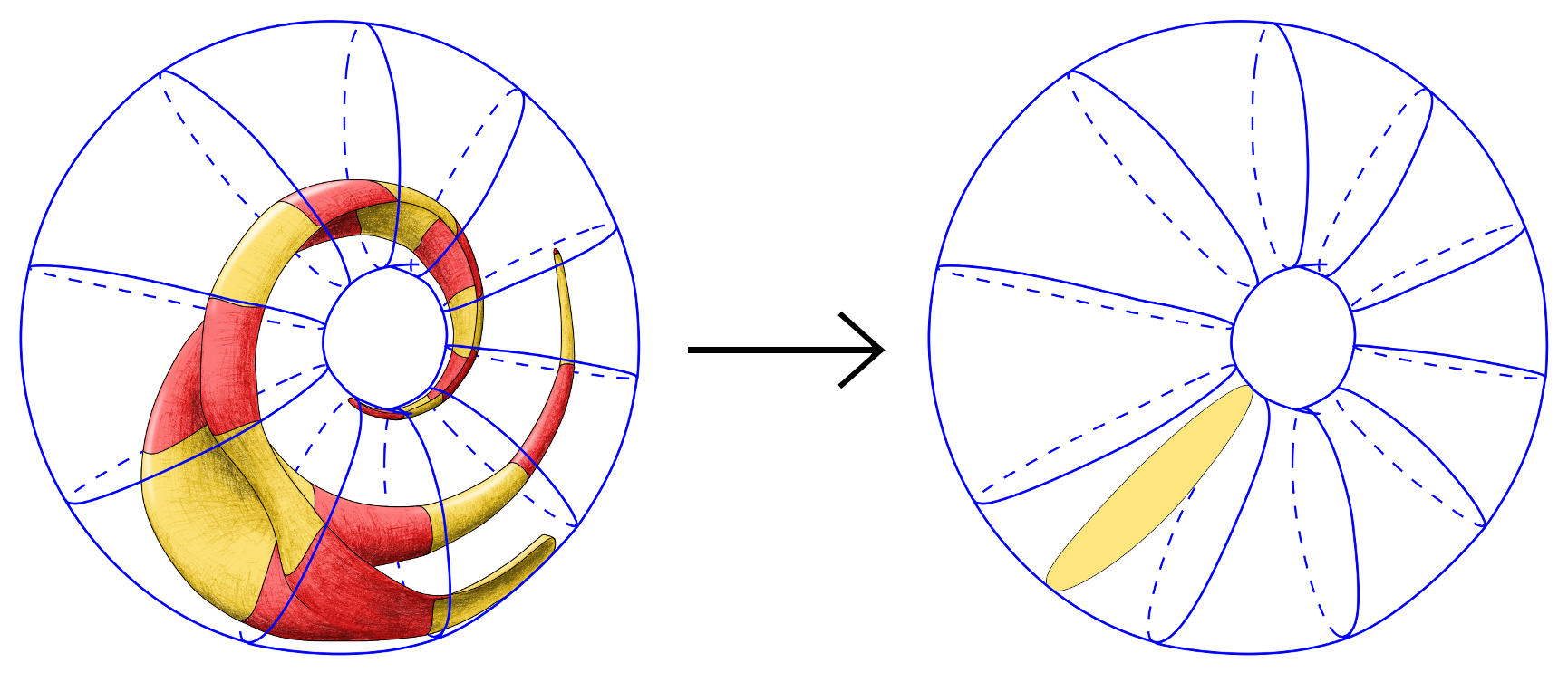}
\caption{A jesters hat initial manifold is taken to a flat disc as $t\rightarrow \infty$.}
\end{figure}
The proof uses an integral iteration technique to obtain the cruicial gradient estimates. In \cite{Huiskengraph}, Huisken uses similar arguments in the case of a cylindrical boundary manifold. The advantage of this method is that boundary curvature is less of an issue; we require boundedness of the derivatives of certain functions as opposed to a sign on them. We remark that this theorem allows some unusual initial manifolds, for example the disc may wrap itself around the inside of the torus several times (see Figure \ref{jestershat}) as long as it is transversal to the vector field generated by the group of rotations.

For this paper we will need various geometric quantities on various manifolds. A bar will imply quantities on $\mbb{R}^{n+1}$, for example $\overline \Delta, \overline \nabla, \ldots$ and so on; no extra markings $\Delta, \nabla, \ldots $ will refer to geometric quantities on $M_t$ the flowing surface at time $t$ and for any other manifold $Z$,  $ \Delta^Z,  \nabla^Z, \ldots \text{etc.}$ will refer to the Laplacian, covariant derivatives, $\ldots$ on $Z$. We will define the volume form on $M_t$ to be $\check{\mu}$ and define $\check{\mu}_\partial$ to be the volume form on $\partial M_t$.


\section{The torus}\label{thetorus}
In this section we make some remarks about $\Sigma$, a torus of revolution. We define the half space $\bb{R}^n_+=\left\{(x_1, \ldots, x_n,0)| x_i \in \bb{R}, x_n>0 \right\}\subset \bb{R}^{n+1}$, where we will sometimes write $\mb{y}=x_1 \mb{e}_1+ \ldots x_{n-1} \mb{e}_{n-1}$ and $x_n=r$. Suppose we have any compact domain $\Omega\subset \bb{R}^n_+$ with smooth boundary $\partial \Omega$ parametrised by $\mb{P}: S^{n-1} \rightarrow \bb{R}^n_+$, then by rotating in the $\{ \mb{e}_n, \mb{e}_{n+1}\}$--plane we define $\Sigma$ to be the torus of revolution $\partial \Omega$ sweeps out. Since $\mbt$ is the direction of the rotation, at a point $p\in \Sigma$ we know that $\mbt(p) \in T_p\Sigma$. 

We will require the values of the second fundamental form of $\Sigma$ in the direction $\mbt$ explicitly. We parametrise $\Sigma$ by
\[\mb{J}(x, \theta) = \mb{P}(x) - \ip{\mb{P}(x)}{\mb{e}_n}\mb{e}_n +\ip{\mb{P}(x)}{\mb{e}_n}\left[ \cos \theta \mb{e}_n + \sin \theta \mb{e}_{n+1} \right]\ \ .\]
If $\nu^P$ is the outward pointing unit normal to $\mb{P}$ in $\bb{R}^n$ then $\mu$, the outward pointing unit normal to $\Sigma$ in $\bb{R}^{n+1}$ is given by
\[\mu=\nu^P - \ip{\nu^P}{\mb{e}_n}\mb{e}_n + \ip{\nu^P}{\mb{e}_n}\left[ \cos \theta \mb{e}_n + \sin \theta \mb{e}_{n+1}\right]\ \ .\]
We may easily see that
\[ \frac{\partial^2 \mb{J}}{\partial x^i \partial \theta} =\ip{\pard{\mb{P}}{x^i}}{\mb{e}_n}\left[ -\sin \theta \mb{e}_n + \cos \theta \mb{e}_{n+1} \right]\ \  \]
is perpendicular to $\mu$. Therefore we know that the direction $\pard{\mb{J}}{\theta}=r\mbt$ is an eigenvector of $A^\Sigma(\cdot, \cdot)$. The eigenvalue may be calculating by writing
\[\frac{\partial^2 \mb{J}}{\partial^2 \theta} = -\ip{\mb{P}}{\mb{e}_n}\left[\cos \theta \mb{e}_n +\sin \theta \mb{e}_{n+1} \right]\]
and so $A^\Sigma(r\mbt, r\mbt)=-\ip{\nu^P}{\mb{e}_n}\ip{\mb{P}}{\mb{e}_n}=-r\ip{\mu}{\mb{r}}$ .  

It will also be useful to consider the flowing manifold as a graph over $\Omega$ in the interior of the torus, when this is possible. At every point $x \in \Omega$ we assign $u$, the angle through which we need to rotate $x$ about $\mb{0}$ to hit the manifold, so that we may parametrise the manifold inside the torus by $\mb{F}: \Omega \rightarrow \bb{R}^{n+1}$,  where
\begin{equation} 
\mb{F}(\mb{x}) = \mb{y} + r(\cos(u) \mb{e}_n + \sin(u) \mb{e}_{n+1})\ \ . \label{paramF} 
\end{equation}

We may now compute all standard geometric quantities with respect to $u$ by standard methods. For example
\[g_{ij} = \delta_{ij} + r^2D_i u D_j u \ \ ,\]
\[\nu = \frac{ -r( \mb{y} +  D_n u (\cos(u) \mb{e}_n + \sin(u) \mb{e}_{n+1})) -\sin(u) \mb{e}_n + \cos(u) \mb{e}_{n+1}}{\widetilde{v}}\ \ ,\]
where we define $\widetilde{v} = \sqrt{\det ( g_{ij} )} = \sqrt{1 + r^2 |Du|^2}$. Similarly we may calculate the equations for mean curvature flow in these coordinates for $M_0$ a manifold which may be parametrised as above by $u_0$. Such a parametrisation will always be possible when $M_0$ is transversal to $\mbt$, although the range of $u_0$ may be more than $2\pi$. Considering graphically, equation (\ref{MCF}) is equivalent to 
\begin{equation}
\begin{cases}\label{pMCF}
u(x, 0)= u_0(x) & \forall x\in \Omega\\
\pard{u}{t} = g^{ij} D_{ij} u +\frac{D_n u}{r}\left(1 +\frac{1}{\widetilde{v}^2}\right)= -\frac{H \widetilde{v}}{r} &\forall x\in \Omega\\
\gamma \cdot Du =0 & \forall x \in \partial \Omega
\end{cases}
\end{equation}
where $\gamma$ is the outward unit normal to $\partial \Omega$ and $g^{ij}$ is the inverse of the metric in this parametrisation. We note that uniform parabolicity of the above is equivalent to the gradient estimate $\widetilde{v}<C< \infty$. We also remark that $\widetilde{v} = \frac{1}{\ip{\nu}{-\sin(u) \mb{e}_n + \cos(u) \mb{e}_{n+1}}} = \frac{1}{\ip{\nu}{\mbt|_\mb{F}}}$.

\section{Evolution equations, boundary derivatives and initial estimates}
Here we will obtain initial estimates on various quantities via a maximum principle of the following form:
\begin{theorem}[Weak Maximum Principle]
\label{WMP}
 Suppose we have a function $f:M^n \times [0,T) \rightarrow \mbb{R}$ then if $f$ satisfies
\begin{equation*}
 \begin{cases}
  \displaystyle \ho f(\mbf p , t) \leq 0 & \forall (\mbf p , t) \in M^n \times [0,T) \text{ such that } \nabla f(\mbf p) =0\\
    \ip{\nabla f }{ \mu } \leq 0 & \forall (\mbf p , t) \in \partial M^n \times [0,T)\\
 \end{cases}
\end{equation*}
then $f(\mbf x ,t ) \leq \underset{\mbf p \in M^n} \sup f ( \mbf p ,0 )$ for all $(\mbf x , t) \in M^n \times [0,T)$.
\end{theorem}

We will repeatedly use the following easily verified relations
\begin{equation} \label{eqdrt}
\ov \n r = \mb{r}, \qquad \ov \n_X \mb{r} = \ip{X}{\mbt} \frac{1}{r} \mbt, \qquad  \ov \n_X \mbt = - \ip{X}{\mbt} \frac{1}{r} \mb{r} \ \ . 
\end{equation}

\begin{lemma}\label{evolu}
 Let $\theta$ be the angle around the torus, taken from some arbitrary base angle. Then
\[ \ho \theta = \frac{2}{r}\ip{\mb{r}^\top}{\n \theta} = -\frac{2}{r^2}\ip{\nu}{\mbt}\ip{\nu}{\mb{r}}\ \ ,\]
and at the boundary
\[ \n_\mu \theta = 0\ \ .\]
\end{lemma}
\begin{proof} 
 Using cylindrical coordinates on $\bb{R}^{n+1}$ we see that $\ov \n \theta = \frac{\mbt}{r}$ and from this we may calculate the evolution equation of $\theta$. We see that
\[\frac{d \theta}{dt} = -\frac{H}{r}\ip{\nu}{\mbt}\ \ ,\]
and
\begin{flalign*} 
\Delta \theta &= g^{ij} \bigg(\ip{-\ip{\mb{r}}{\pard{}{x^i}}\frac{\mbt}{r^2} -\ip{\pard{}{x^i}}{\mbt} \frac{\mb{r}}{r^2}}{\pard{}{x^j}} - \ip{\frac{\mbt}{r}} {h_{ij}\nu}\bigg)\\ 
&= - \frac{2}{r^2}\ip{\mb{r}^\top}{\mbt^\top} -H \ip{\frac{\mbt}{r}}{\nu}\ \ .
\end{flalign*}
Therefore
\begin{equation*}
\ho \theta = \frac{2}{r^2} \ip{\mb{r}^\top}{\mbt^\top}=\frac{2}{r}\ip{\mb{r}^\top}{\n \theta}\ \ .
 \end{equation*}

At the boundary we have that $\mbt (\mb{p}) \in T_\mb{p} \Sigma$, and therefore
\[\ip{\n \theta}{\mu}=\ip{\left( \frac{\mbt}{r} \right)^\top}{\mu}=\ip{\frac{\mbt}{r} - \frac{\ip{\mbt}{\nu}}{r}\nu}{\mu}=0\ \ .\]
\end{proof}
Since the above only depends on the \emph{derivatives} of $\theta$, the same equations hold for manifolds that wrap themselves around the torus more than once (as in Figure \ref{jestershat}) by simply extending the range of $\theta$ to be more than $2\pi$. The function $u$ defined in the previous section is an example of this, and \emph{on the flowing manifold} $u$ will satisfy the same evolution equation as $\theta$. From now on we use the extended $\theta$.

\begin{cor}\label{thetabound}
The function $\theta(\cdot, t)$ is bounded above by its initial supremum and below by its initial infimum.
\end{cor}
\begin{proof}
This follows directly from the above maximum principle.
\end{proof}
Therefore, the disk may not twist itself around the torus any more than it is twisted initially. The following will be required later:
\begin{lemma}\label{evolr}
 The function $r$ evolves by
\[ \ho r = -\frac{|\mbt^\top|^2}{r} \]
while at the boundary
\[ \n_\mu r = \ip{\mu}{\mb{r}}\ \ .\]
\end{lemma}
\begin{proof}
Similarly to Lemma \ref{evolu} we have $\ddt{r} = -H\ip{\mb{r}}{\nu}$ and calculate
\begin{flalign*}
 \Delta r &=g^{ij} \left( \pard{}{x^i} \ip{\mb{r}}{\pard{}{x^j}} - \ip{\mb{r}}{\n_\pard{}{x^i} \pard{}{x^j}} \right)\\
&=g^{ij} \left( \ip{\frac{1}{r}\ip{\pard{}{x^i}}{\mbt} \mbt}{\pard{}{x^j}} + \ip{\mb{r}}{\ov \n_\pard{}{x^i} \pard{}{x^j} - \n_\pard{}{x^i} \pard{}{x^j}}\right)\\
&=-H \ip{\nu}{\mb{r}}+\frac{1}{r}|\mbt^\top|^2\ \ .
\end{flalign*}
\end{proof}
 \begin{remark}\label{rbound}
 We will assume throughout that $r_0 \leq r \leq r_1$, a consequence of the flowing manifold staying within the torus. Certainly this will be true for all time that we have a gradient estimate, and apriori will be true for $t \in[0, \epsilon)$ for some small $\epsilon>0$.
 \end{remark}
We will need the following well known evolution equations:
\begin{lemma}\label{evols}
On the interior of a manifold moving by mean curvature flow the following hold 
\begin{flalign}
\frac{d \nu}{dt} &= \n H\\
 \frac{d g_{ij}}{dt} &= -2Hh_{ij} \label{dgijdt}\\
\ho H &= H|A|^2
\end{flalign}
\end{lemma}
\begin{proof}
 See for example \cite[Lemma 3.2, Lemma 3.3 and Corollary 3.5]{Huiskenconvex} .  
\end{proof}
The boundary derivative of $H$ and relations on the second fundamental forms of the flowing manifold and the boundary manifold for equation (\ref{MCF}) were calculated by Stahl \cite[Proposition 2.2, 2.4]{Stahlsecond}, and are summarised in the following Lemma:
\begin{lemma}\label{thankyoustahl}
At the boundary 
\[\n_\mu H = A^\Sigma(\nu, \nu)\]
and for $X \in T_\mb{p} M \cap T_\mb{p} \Sigma$ then
\[A^\Sigma(X, \nu) +A(X, \mu)=0\ \ .\] 
\end{lemma}
 
We now define the gradient function, $v= \frac{1}{\ip{r \mbt}{\nu}}$. Without loss of generality we may assume that this is positive and so may define the related function $Q=\log{v}$.
\begin{lemma}\label{evolvQ}
 The evolution equations for $Q$ and $v$ (while they are finite) are
\begin{flalign*} 
\ho Q &= -|A|^2 - |\n Q|^2\ \ ,\qquad  \ho v = -v|A|^2 - \frac{2}{v}|\n v|^2\ \ ,
\end{flalign*}
and the boundary derivatives are
\[\n_\mu Q = -A^\Sigma(\nu, \nu)\ \ , \qquad\qquad \n_\mu v = -v A^\Sigma(\nu, \nu) \ \ .\]
\end{lemma}
\begin{proof}
 We will first calculate the evolution of $w=\ip{r \theta}{\nu}$. 

Using (\ref{eqdrt}) we may immediately see that
\[ \ddt w = -H \ip{\ov \n_\nu (r \mbt)}{\nu} + \ip{r \theta}{\n H} = \ip{r \theta}{\n H}\ \ .\]

Writing $Z=r \mbt$ we have 
\begin{flalign*}
\Delta \ip{\nu}{Z} &=g^{ij} \ip{ \ov \n_i (\ov \n_j \nu) - \ov \n_{\n_i \pard{}{x^j} }\nu}{Z} \\
&\qquad\qquad\qquad+ 2 g^{ij}\ip{\ov \n_i \nu}{\ov \n_j Z } + g^{ij} \ip{\nu}{ \ov \n^2_{ij} Z - H \ov \n_\nu Z } \ \ .
\end{flalign*}

For the first of these terms, take a orthonormal basis $\{\mb{f}_1, \ldots, \mb{f}_n\}$ at a point $\mb{p}\in M$. We extend this to give orthogonal geodesic coordinates at $\mb{p}$. We calculate that at $\mb{p}$,
\begin{flalign*}
 g^{ij} \ip{ \ov \n_i (\ov \n_j \nu) - \ov \n_{\n_i \pard{}{x^j} }\nu}{Z}&=g^{ij}\ip{\mb{f}_j(\mb{f}_i \nu)}{Z}\\
&=g^{ij}\n_jh_{il}g^{lk}\ip{\mb{f}_k}{Z}-g^{ji}h_{il}g^{lk}\ip{h_{jk}\nu}{Z}\\
&=\n_{Z^\top} H -\ip{\nu}{Z}|A|^2
\end{flalign*}
where we used the Weingarten and Codazzi equations. Since the right hand side does not depend on the coordinate system this holds for all $\mb{p}\in M$.

For the second term we have
\[g^{ij}\ip{\ov \n_i \nu}{\ov \n_j Z} = g^{ij}\ip{\ov \n_i \nu}{\ip{\mb{r}}{\pard{}{x^j}} \mbt^\top - \ip{\pard{}{x^j}}{\mbt}\mb{r}^\top} =0\ \ .\]

The final term also vanishes; we may see that 
\begin{flalign*}
 \ov \n^2_{XY} Z &= \ov \n_Y (\ip{X}{\mb{r}} \mbt - \ip{X}{\mbt} \mb{r}) - \ov \n_{\ov \n_X Y } Z\\
&= \frac{\ip{Y}{\mbt}}{r}\left[\ip{X}{\mbt}\mbt - \ip{X}{\mb{r}}\mb{r} +\ip{X}{\mb{r}}\mb{r} - \ip{X}{\mbt} \mbt\right] =0\ \ .
\end{flalign*}
Therefore
\[\ho w = w |A|^2\ \ ,\]
and the evolution equations for $v$ and $Q$ immediately follow.

At the boundary using Lemma \ref{thankyoustahl} and the fact that $\mu \perp \mbt$ we have 
\[ \n_\mu \ip{\nu}{Z} = A(\mu, Z^\top) + \ip{\nu}{\ip{\mb{r}}{\mu}\mbt - \ip{\mu}{\mbt} \mb{r}}=\ip{\nu}{\mbt}\ip{\mu}{\mb{r}}-A^\Sigma(\nu,Z- w \nu) \ \ .\]
Now from calculations in Section \ref{thetorus} we know that $Z$ is an eigenvalue of $A^\Sigma(\cdot, \cdot)$, and using the calculation of this eigenvalue we see
\[\n_\mu \ip{\nu}{Z}=\ip{\nu}{\mbt}\ip{\mu}{\mb{r}} +\frac{1}{r}\ip{\nu}{\mbt}A^\Sigma(Z,Z)+wA^\Sigma(\nu, \nu)=wA^\Sigma(\nu, \nu)\]
and we are done.
\end{proof}

\begin{cor} \label{Hbound}
 While $v$ is bounded we have the estimate $H^2<C_H$ for some constant $C_H>0$ depending only on $M_0$ and $\Sigma$ (and independant of $v$).
\end{cor}
\begin{proof}
We calculate 
\begin{flalign*}
 \ho H^2v^2 &=-2v^2|\n H|^2- 8Hv\ip{\n H}{\n v} -6H^2|\n v|^2\\
&=-2\left[v^2|\n H|^2+ Hv\ip{\n H}{\n v}\right] \\
&\qquad\qquad\qquad\qquad\qquad- 6\left[H^2|\n v|^2 +Hv\ip{\n H}{\n v}\right]\ \ .
\end{flalign*}
At a positive stationary point we have $H\n v = -v \n H$, and so the above vanishes and we may apply our maximum principle. Since $\frac{1}{v}<r$ we have, using Remark \ref{rbound}, a uniform bound.
\end{proof}
\begin{remark}
We may attempt to get a positive lower bound on $H$ in a similar way and indeed the evolution equations are amenable. In fact due to the boundary condition, this is not useful. Using the vector field $Z$ as in Lemma \ref{evolvQ} we see
\[ 0=-\int_{\partial M} \ip{\mu}{Z} \dmu_\partial = \int_M \text{div}(Z^\top) \dmu = \int_M \frac{H}{v} \dmu \]
where again we used equation (\ref{eqdrt}). Since by assumption $v$ is initially positive, $H$ cannot be initially positive everywhere, and therefore (for example) weak mean convexity implies we must have a minimal hypersurface -- not much left for the flow! Indeed a corollary of our main theorem is that the only such minimal hypersurface that satisfies our initial conditions is the flat profile.
\end{remark}

\section{Integral estimates}
As a prerequisite to applying the Stampaccia iteration method as in \cite{Huiskengraph} we now give some of the required boundary estimates. In particular, we modify various Lemmas for graphs with boundary in \cite{Gerhardt} to manifolds with boundary. 

For a start, we will require the Michael--Simon--Sobolev inequality from \cite{MSIneq}. While this holds in much more general situations, we will only require $M$ to be smooth embedded $n$-dimensional manifolds in $\bb{R}^{n+1}$. 
\begin{lemma}[The Michael--Simon--Sobolev inequality]
 There exists a constant $C_S>0$ depending only on $n$ such that for any function $f\in C^1(\overline{M})$ such that $f$ has compact support, we have
\begin{flalign*}
 \left( \int_M |f|^\frac{n}{n-1} \dmu \right)^\frac{n-1}{n} \leq C_S \int_M |\n f| + |H| |f| \dmu \ \ .
\end{flalign*}
\end{lemma}
We also need such an inequality not just on functions of compact closure, but functions that may be non-zero at the boundary $\partial M$.
\begin{lemma}
 For any compact manifold $M$ with boundary $\partial M$ and for any function $f\in C^1(\overline{M})$ we have
\begin{flalign*}
 \left( \int_M |f|^\frac{n}{n-1} \dmu \right)^\frac{n-1}{n} \leq C_S\left[ \int_M |\n f| + |H| |f| \dmu + \int_{\partial M} |f| \dmu_\partial \right]
\end{flalign*}
where the constant $C_S$ depends only on $n$.
\end{lemma}
\begin{proof}
This is as in \cite[Lemma 1.1]{Gerhardt}. Let $d: D \rightarrow \bb{R}$ be the function giving the minimum distance \emph{along the manifold} to the boundary. This is smooth close enough to the boundary. We define for $k$ large enough $\widetilde{\eta}_k = \min\{ 1, k d\}$, and let $\eta_k$ be a $C^1$ smoothing of this. We consider the sequence $f_k=\eta_k f$ for $k \in \bb{N}$. Since $\check{\mu}(\{ x | f(x) \neq f_i\})\rightarrow 0 $ as $i \rightarrow \infty$ we have that
\[\left( \int_M |f_k|^\frac{n}{n-1} \dmu \right)^\frac{n-1}{n} \rightarrow \left( \int_M |f|^\frac{n}{n-1} \dmu \right)^\frac{n-1}{n}, \ \ \int_M  |H| |f_k| \dmu \rightarrow \int_M  |H| |f| \dmu \ \ .\]
We also see,
\begin{flalign*}
\int_M |\n f_k| \dmu &\leq \int_M |\n f|\eta_k \dmu +\int_M |f| |D \eta_k| \dmu\ \ .
\end{flalign*}
The first term of the above may be estimated similarly to the other terms. For the final term we choose a special parametrisation of the collar. We parametrise by $\mb{C}: \partial M \times [0, \epsilon)\rightarrow \bb{R}^{n+1}$ for $\epsilon$ small enough by setting  $\mb{C}(x,\epsilon)$ to be the point obtained by starting at $x\in \partial M$ and moving distance $\epsilon$ down the geodesic starting at $x$ with direction $-\mu$. Therefore $\pard{}{x^n} = \n d$ and the metric induced by $\mb{C}$ has $g_{in} = \delta_{in}$. Therefore for $k$ large enough
\begin{flalign*}
\int_M |f| |D \eta_k| \dmu &\leq \int_{\{x\in M| d(x)\leq \frac{1}{k}\}} k |f| \dmu\\
&=k\int_0^\frac{1}{k} \int_{\partial M^n} |f|\sqrt{\det(g_{ij}(x,s))} \dmu_\partial ds\\
&\rightarrow \int_{\partial M^n} |f| \sqrt{\det(g_{ij}(x,0))} \dmu_\partial= \int_{\partial M^n} |f| \dmu_\partial
\end{flalign*}
as $k\rightarrow \infty$. 
\end{proof}

Clearly we need to estimate boundary integrals, and we now give one way of doing so, based on \cite[Lemma 1.4]{Gerhardt}.
\begin{lemma} \label{Mbndryint}
 For $M$ a compact manifold with boundary, then  for all $f\in W^{1, \infty}(M)$ we have
\[ \int_{\partial M} |f| \dmu_\partial \leq C_\Sigma \int_M |\n f| +(|H|+1)|f| \dmu\]
where the constant $C_\Sigma>0$ depends only on $\Sigma$.
\end{lemma}
\begin{proof}
 This is essentially just divergence theorem. We now use $\ov d$, the minimum distance to $\Sigma$ in $\bb{R}^{n+1}$ and note that at $\Sigma$, $\ov \n \,\ov d = -\mu$. We take a smooth function $\phi: \bb{R} \rightarrow \bb{R}$ such that $\phi'(0)=-1$ and $\phi(x)=0$ for $x>R$ where $R$ is less than the minimum focal distance of $\Sigma$. We define $\ov \phi = \phi(\ov d)$ -- a smooth function on $\bb{R}^{n+1}$. Then 
\[\Delta \ov{\phi} =g^{ij} \ov \n^2 \ov \phi \left(\pard{\mb{F}}{x^i}\right)\left(\pard{\mb{F}}{x^j}\right) - H \ip{\nu}{\ov \n \ov \phi} \leq C_1(1 + |H|)\ \ .\]
for some $C_1>0$ depending on the derivatives of $\ov \phi$ and so
\begin{flalign*}
 \int_{\partial M} f \dmu_\partial &= \int_M \text{div}(f\n \ov{\phi})\dmu \\
&= \int_M \ip{\n f}{\n \ov{\phi}} + f \Delta \ov{\phi} \dmu \\
&\leq C_2 \int_M |\n f| + f(|H|+1) \dmu\ \ ,
\end{flalign*}
for some $C_2>0$ depending on the derivatives of $\ov \phi$.
\end{proof}
\begin{cor}
\label{fullsobolev}
 For all $f\in C^1(\overline{M})$ there exists a constant $\ov{C}_S$ depending on $n$ and $\Sigma$ such that
\[ \left( \int_M |f|^\frac{n}{n-1} \dmu \right)^\frac{n-1}{n} \leq \ov{C}_S \int_M |\n f| + (|H|+1) |f| \dmu \]
\end{cor}

Due to the boundary condition we have the following:
\begin{lemma}
\label{dintdt}
Suppose $f: M^n \times[0,T)\rightarrow \bb{R}$ is once differentiable in time such that $\ddt{f}, f \in L^1(M_t)$. Then the following holds for $t>0$ and $\beta=0$:
 \begin{equation}\frac{d}{dt} \int_{M_t} f \dmu = \int_{M_t} \frac{d f}{dt} - H^2 f \dmu  
  \label{ddtint}
 \end{equation}
\end{lemma}
\begin{proof}
The perpendicular boundary condition implies that the manifold does not flow out through the boundary. Therefore, we know that in the parametrisation defined by $\mb{F}$ in (\ref{MCF}) over the stationary domain $M^n$ that
\[\int_{M_t} f \dmu = \int_{M^n} f \sqrt{\det(g_{ij}(x,t)) }dx\ \ .\]
Now the equation follows from (\ref{dgijdt}). 
\end{proof}

\begin{remark}\label{L2H}
Integrating equation (\ref{dintdt}) for $f=1$ with respect to time and rearranging we see that
\[\int_0^T \int_{M_t} H^2 \dmu dt \leq |M_0|\ \ ,\]
that is we have a parabolic $L^2$ estimate on $H$ that does not depend on the time interval. 
\end{remark}

We will also require the following well known Lemma, which serves to streamline the iteration argument of the next chapter:
\begin{lemma}\label{Stampacciait}
 Suppose $\phi:(k_0, \infty) \rightarrow \bb{R}$ is a non--negative non--increasing function such that for all $h>k\geq k_0$ then
\[ \phi(h) \leq \frac{C}{(h-k)^\alpha} (\phi(k))^\beta\]
where $C, \alpha$ and $\beta$ are positive constants. Then if $\beta>1$ then $\phi(k_0+d)=0$ for
\[ d^\alpha = C [\phi(k_0)]^{\beta-1} 2^{\alpha\frac{\beta}{\beta-1}}\ \ .\]
\end{lemma} 
\begin{proof}
See \cite[Lemma 4.1 i)]{Stampacchia}, for example.
\end{proof}

\section{The gradient estimate via iteration}

Here we will give a bound on the gradient $Q=\log(-\ip{\nu}{ r \mbt})$ via integral estimates. We define $Q_k=(Q-k)_+$ and $A(k) =\{ x\in M | Q_k>0 \}$ and we aim to get suitable estimates on the quantity
\[ \|A(k)\|= \int_0^T\int_{A(k)} \dmu \, dt\ \ ,\]
ultimately showing that this is zero for all sufficiently large $k$. We begin with some $L^p$ estimates on $Q_k$.

\begin{lemma}
\label{Qkpestimate}
There exsits a $k_1>0$ such that $k>k_1$ and even $p$ 
\[ \int_0^T\int_M Q_k^p \dmu\,dt \leq C_Q(p) \|A(k)\|\]
where $C_Q>0$ depends on $p, n, \Sigma, M_0$ and $k_1$.
\end{lemma}
\begin{proof}
Using Lemma \ref{evolu} and writing $w=\ip{r\mbt}{\nu}$ we have that 
\begin{flalign*}
 \ho e^{\lambda \theta} &= \lambda e^{\lambda \theta}\left[ -\frac{2}{r^2}\ip{\nu}{\mbt}\ip{\nu}{\mb{r}} - \lambda |\n \theta|^2 \right]\\
&\leq \lambda e^{\lambda \theta} \left[ \frac{w}{r^3} C^\theta_1 -\frac{\lambda}{r^2} \left(1-\frac{w^2}{r^2}\right) \right]
\end{flalign*}
for $C_1^\theta>0$. At the boundary this function has zero derivative in the $\mu$ direction (from Lemma \ref{evolu}). On $A(k)$, we may estimate $ w<e^{-k}$ and so recalling Remark \ref{rbound} and writing $|\n \theta|^2 = \frac{1}{r^2} - \frac{w^2}{r^4}$, we may get a positive lower bound on $|\n \theta|$ for $k$ sufficiently large. Therefore, we may choose a large enough $\lambda, k_0>0$ so that for all $k>k_0$ on $A(k)$ 
\begin{flalign*}
 \ho e^{\lambda \theta} &\leq - C^\theta_2 |\n \theta |^2\\
&\leq -3C_3^\theta
\end{flalign*}
holds, for $C^\theta_2,C^\theta_3>0$ depending on $\lambda$, where we used our bounds on $\theta$ from Corollary \ref{thetabound}.
We now agree to write $C_n$ for any bounded positive constant depending only on $M_0$, $p$, $\Sigma$, $k_0$ and $n$. Using Lemma \ref{evolvQ} we calculate that on $A(k)$ for $k>k_0$ and $p>2$,
\begin{flalign*}
\ho  e^{\lambda \theta} Q_k^p &\leq Q_k^{p-2} [ -C_2^\theta Q_k^2|\n \theta|^2 - p e^{\lambda \theta}  Q_k|\n Q|^2 \\
&\qquad\qquad\qquad\qquad\quad- e^{\lambda \theta} p(p-1)|\n Q|^2 + C_n pQ_k|\n Q||\n \theta|]\\
&\leq Q_k^{p-2}\left[C_n Q_k-C_2^\theta Q_k^2|\n \theta|^2 - C_nQ_k|\n Q|^2- C_n|\n Q|^2 \right]\\
&\leq C_n Q_k^{p-2} -2C_3^\theta Q_k^p -C_nQ_k^{p-1} |\n Q|^2 -C_n Q_k^{p-2}|\n Q|^2\ \ ,
\end{flalign*}
where we used the lower bound on $|\n\theta|$ and Young's inequality of the form $ab = \frac{\epsilon a^2}{2} + \frac{b^2}{2\epsilon}$ repeatedly. We may now use Lemma \ref{dintdt} and divergence theorem to see that
\begin{flalign*}
 \frac{d}{dt}\int_M e^{\lambda \theta}Q_k^p \dmu &\leq \int_M \ho e^{\lambda \theta} Q_k^p - H^2 e^{\lambda \theta} Q_k^p\dmu \\
&\qquad\qquad\qquad\qquad+ \int_{\partial M} p C_n C_\Sigma Q_k^{p-1} \dmu_\partial\ \ .
\end{flalign*}
Estimating as above and using Lemma \ref{Mbndryint} and Corollary \ref{Hbound},
\begin{flalign*}
 \frac{d}{dt}&\int_M e^{\lambda \theta}Q_k^p \dmu\\
&\leq \int_M Q_k^{p-2}\left[C_n -2C_3^\theta Q_k^2 -C_nQ_k |\n Q|^2 -C_n |\n Q|^2 +C_nQ_k + C_n|\n Q| \right] \dmu\\
&\leq \int_M Q_k^{p-2}\left[C_n - C_3^\theta Q_k^2 \right] \dmu\ \ .
\end{flalign*}
Hence choosing $k>k_1=\max\{ k_0, \underset{M_0} \sup Q +1\}$ we may integrate to get
\begin{flalign*}
 \int_0^T \int_{A(k)} Q_k^p \dmu\, dt \leq \widetilde{C} \int_0^T \int_{A(k)} Q_k^{p-2} \dmu \,dt\ \ .
\end{flalign*}

It is easily verified that the above argument also holds if $p=2$. Therefore, by induction we may estimate for $p$ even and $k>k_1$
\[\int_0^T\int_M Q_k^p \dmu\,dt \leq C_Q(p) \|A(k)\|\ \ .\]
\end{proof}

In addition to the above we also need an estimate on $\|A(k)\|$ which does not depend on $T$. This comes about by using Remark \ref{L2H} to estimate the highest order terms in a Laplacian after careful use of the boundary condition:

\begin{prop}\label{thefinalestimate}
 There exists a $k_2>0$ such that for all $k>k_2$ there exists a constant $C$ depending only on $M_0, \Sigma$ and $n$ such that
\[\|A(k)\|\leq C\]
\end{prop}
\begin{proof}
 We consider the bounded function $\kappa=\theta^2r^2$. We will calculate the time derivative of the integral of $f$ over the manifold. From Lemmas \ref{evolu} and \ref{evolr} we see
\begin{flalign*}
 \ho \kappa &=r^2\left[ \frac{4\theta}{r}\ip{\n r}{\n \theta} - 2|\n \theta |^2 \right] + \theta^2\left[ -2 |\mbt^\top|^2 - 2|\mb{r}^\top|^2 \right]\\
&\qquad\qquad\qquad\qquad\qquad\qquad\qquad\qquad\qquad -8 \theta r \ip{\n r}{\n \theta}\ \ .
\end{flalign*}
At the boundary we have $\n_\mu \kappa = \theta^2\n_\mu r^2$ and so calculate
\begin{flalign*}
\int_{\partial M} \n_\mu \kappa\dmu_\partial &= \int_M \text{div}(\theta^2 \n r^2) \dmu = \int_M \ip{\n \theta^2}{ \n r^2} + \theta^2 \Delta r^2 \dmu\\
&=\int_M  4\theta r\ip{\n \theta}{ \n r} + \theta^2 \left[-2rH \ip{\nu}{\mb{r}}+2|\mbt^\top|^2 +2|\mb{r}^\top|^2\right] \dmu\ \ .
\end{flalign*}
Therefore, by divergence theorem, 
\begin{flalign*}
 \ddt{}\int_M \kappa \dmu &= \int_M \ho \kappa - H^2 \kappa \dmu + \int_{\partial M} \n_\mu \kappa \dmu\\
&=\int_M  - 2r^2|\n \theta |^2 -2\theta^2rH\ip{\nu}{\mb{r}} -H^2 \kappa \dmu\ \ .
\end{flalign*}
We note that $\ip{\nu}{\mb{r}}^2+\ip{\nu}{\mbt}^2 \leq |\nu|^2 = 1$ and so $\ip{\nu}{\mb{r}}^2 \leq 1-\ip{\nu}{\mbt}^2 = |\mbt^\top|^2$. Since $|\n \theta |^2 = \frac{|\mbt^\top|^2}{r^2}$, using Young's inequality we see
\begin{flalign*}
 \ddt{}\int_M \kappa \dmu &\leq \int_M -2 |\mbt^\top|^2+ 2(\theta^2 r |H|)( |\ip{\nu}{\mb{r}}|) - H^2 f \dmu \\
&\leq \int_M  - |\mbt^\top|^2+ C_1 H^2 \dmu \ \ .
\end{flalign*}
for some $C_1>0$ by the boundedness of $r$ and $\theta$.

Now integrating with respect to time as in Remark \ref{L2H} and using the $L^2$ bound on $H$,
\[\int_0^T\int_M |\mbt^\top|^2 \dmu \,dt \leq C_1 \int_0^T \int_M H^2 \dmu \,dt + \int_{M_0} \theta^2 r^2 \dmu\Big|_{t=0} \leq C_3\]
for some constant $C_3>0$ depending on the bounds on $\theta^2$, $r^2$ and $|M_0|$ but not on $T$. On the region $A(k)$,  $\ip{\nu}{\mbt} \leq \frac{1}{r} e^{-k}$ and so choosing $k_2$ large enough that $\ip{\nu}{\mbt}\leq \frac{1}{\sqrt{2}}$ then 
\begin{equation} \|A(k)\| \leq 2\int_0^T\int_M |\mbt^\top|^2 \dmu \, dt \leq 2C_3\ \ . \label{sgr}
\end{equation}
\end{proof} 

We now put these together to give the gradient estimate.
\begin{theorem}[Gradient Estimate]\label{generalgradest}
 There exists a $C_Q>0$ depending only on $n$, $\Sigma$ and $M_0$ such that for all time $Q\leq C_Q$. 
\end{theorem}
\begin{proof} By Lemma \ref{evolvQ} we may calculate for $p>2$ that 
\begin{flalign*}
 \ho Q_k^p  &\leq  pQ_k^{p-1}\ho Q - p(p-1) Q_k^{p-2}|\n Q|^2\\
&\leq -pQ_k^{p-2}\left( Q_k |\n Q|^2  +(p-1)|\n Q|^2 \right)\ \ .
\end{flalign*}
Using $C_n$ as in Lemma \ref{Qkpestimate} we see using the bound on $|H|$ and Lemma \ref{Mbndryint} then
\begin{flalign*}
 \frac{d}{dt}\int_M Q_k^p\dmu &\leq \int_M Q_k^{p-2}\Big[ -pQ_k |\n Q|^2  -p(p-1)|\n Q|^2 -H^2Q_k^2\\
&\qquad\qquad\qquad\qquad\qquad\qquad\qquad\qquad\quad\ \, +C_n Q_k +C_n |\n Q| \Big] \dmu\\
&\leq \int_M Q_k^{p-2}\left[C_n Q_k^2 +C_n\right] \dmu - C_n \int_M pQ_k^{p-1}|\n Q| + (|H|+1)Q_k^p \dmu\\
&\leq C_2 \int_M  Q_k^{p-2}+ Q_k^{p} \, \dmu - C_1\left[ \int_M Q_k^{\frac{np}{n-1}} \dmu \right]^\frac{n-1}{n}\ \ .
\end{flalign*}
where on the last line we used Corollary \ref{fullsobolev} with $f=Q_k^p$. Choosing $k>k_3=\max\{\underset{x\in M_0}\sup Q, k_1, k_2\}$ and integrating with respect to time 
\[\underset{t \in [0,T]} \sup \int_M Q_k^p \dmu + C_1\int^T_0  \left[ \int_M Q_k^{\frac{np}{n-1}} \dmu \right]^\frac{n-1}{n} dt\leq C_2\int_0^T\int_M  Q_k^{p-2}+ Q_k^{p} \dmu \, dt\ \ .
\]
We now deal with the left hand side as usual (see \cite{Huiskengraph}), and so after repeated use of the H\"older inequality we have for even $p$,
\begin{flalign*}
\left( \int_0^T \int_M Q_k^{\frac{p(n+1)}{n}} \dmu \, dt \right)^\frac{n}{n+1}&\leq C_n \int_0^T\int_M  Q_k^{p-2}+ Q_k^{p} \dmu \, dt\\
&\leq C_3 \|A(k)\|\ \ ,
\end{flalign*}
where we used Lemma \ref{Qkpestimate}. For $k_3<h<k$, the H\"older inequality now implies
\[|h-k|^p\|A(h)\|\leq \int_0^T\int_M Q_k^{p} \dmu\,dt \leq C_3 \|A(k)\|^{2-\frac{n}{n+1}}\ \ .\]
We may now apply Lemma \ref{Stampacciait} to get that $\|A(k)\|=0$ for $k=k_3+D$ where \[D^p=C_3 2^{2+n} \|A(k_1)\|^{\frac{1}{n+1}}\ \ .\]
Setting $p=4$, by Proposition \ref{thefinalestimate} the Theorem is proved.
\end{proof}

\section{Long time existence and convergence}

We now state and prove the main Theorem.

\begin{theorem}\label{PStorustheorem}
 Suppose $\Sigma$ is a torus of rotation, $M_0$ is a manifold satisfying the boundary condition that nowhere contains the vector field $\mbt$ in its tangent space. Then a solution to equation (\ref{MCF}) with initial data $M_0$ exists for all time and converges uniformly to a flat cross-section of the torus.
\end{theorem}
\begin{proof}
We take $\Omega$ to be a cross-section of the torus $\Sigma$ and rewrite the manifold as a graph, $u_0$, over the cross-section as in section \ref{thetorus} so that the manifold may be parametrised by equation (\ref{paramF}). At a point on the flowing manifold, this will be equal to the function $\theta$. As noted in section \ref{thetorus}, for both uniform parabolicity of equation (\ref{pMCF}) and a gradient estimate on $u$ we need to bound the function $\widetilde{v}=\frac{1}{\ip{\mbt}{\nu}}$. We also note that while $\widetilde{v}$ is finite we may write $M_t$ as a graph. 

Since $\widetilde{v}=rv$, Theorem \ref{generalgradest} gives the upper bound $\widetilde{v} \leq C$, and so we have both uniform parabolicity and a gradient estimate. Corollary \ref{thetabound} also gives $C^0$ bounds on $u$. Therefore, by standard methods we have existence for all time. For example since equation (\ref{pMCF}) has \emph{linear} boundary conditions, with trivial modifications we may apply the arguments of \cite[Section 8.2 and Chapter 12]{Lieberman}.

For convergence we consider integrals of the derivatives of the graph over $\Omega$. We have $\ddt{u}=-Hv$ and so using our gradient estimate and Corollary \ref{L2H},
\[ \int_0^T \int_\Omega \left(\ddt{u}\right)^2 dx\, dt =  \int_0^T \int_M H^2v^2 \dmu\, dt \leq C_1  \int_0^T \int_M H^2 \dmu\, dt\leq C_2\]
where $C_1, C_2>0$ are constants independant of $T$. 

We see that in coordinates $|\mbt^\top|^2 = 1- \ip{\nu}{\mbt}^2= \frac{\widetilde{v}^2-1}{\widetilde{v}^2}=\frac{|Du|^2}{v^2}$. Therefore using the gradient estimate again
\[ \int_0^T \int_\Omega |Du|^2 dx\,dt \leq C_3 \int_0^T\int_\Omega \frac{r|Du|^2}{v} \widetilde{v} dx \, dt = C_3 \int_0^T \int_M |\mbt^\top|^2 \dmu\, dt \leq C_4\]
for constants $C_3, C_4 >0$ where we used equation (\ref{sgr}). 

Therefore there exists a constant $C>0$ such that
\[ \int_0^\infty \int_\Omega \left(\ddt{u}\right)^2+ |Du|^2 dx\,dt \leq C\ \ .\]
Writing $u_\Omega(t)=\frac{1}{|\Omega|}\int_\Omega u(x, t)dx$ for the integral average of $u$ at any time then by the Poincar\'e inequality, the above is enough to ensure that $u(x,t)\rightarrow u_\Omega(t)$ uniformly as $t\rightarrow \infty$. Since we also have from Corollary \ref{thetabound} that $\underset{x\in \Omega}\inf \, u(x,t)$ is nondecreasing and $\underset{x\in \Omega}\sup\, u(x,t)$ is nonincreasing then in fact $u(x,t)$ converges uniformly to a constant as $t\rightarrow \infty$. This corresponds to uniform convergence of $M_t$ to a flat cross-section of the torus.
\end{proof}

\bibliographystyle{plain}
\bibliography{$HOME/dos/bibblee/bib}

\begin{thebibliography}{10}

\bibitem{AltschulerWu}
Steven~J. Altschler and Lang~F. Wu.
\newblock Translating surfaces of the non-parametric mean curvature flow with
  prescribed contact angle.
\newblock {\em Calculus of Variations and Partial Differential Equations},
  2:101--111, 1994.

\bibitem{Buckland}
John~A. Buckland.
\newblock Mean curvature flow with free boundary on smooth hypersurfaces.
\newblock {\em Journal f{\"u}r die Reine und Angewandte Mathematik},
  586:71--91, 2005.

\bibitem{Gerhardt}
Claus Gerhardt.
\newblock Global regularity of the solutions to the capillarity problem.
\newblock {\em Annali della Scuola Normale Superiore Pisa, Classe di Scienze
  $4^{\text{e}}$ s\'erie}, 3:157--175, 1976.

\bibitem{Huiskenconvex}
Gerhard Huisken.
\newblock Flow by mean curvature of convex surfaces into spheres.
\newblock {\em Journal of Differential Geometry}, 20:237--266, 1984.

\bibitem{Huiskengraph}
Gerhard Huisken.
\newblock Non-parametric mean curvature evolution with boundary conditions.
\newblock {\em Journal of Differential Equations}, 77:369--378, 1989.

\bibitem{MSIneq}
Leon M.~Simon James H.~Michael.
\newblock Sobolev and mean value inequalities on generalised submanifolds in
  $\mathbb{R}^n$.
\newblock {\em Communications on Pure and Applied Mathematics}, 26:361--379,
  1973.

\bibitem{LambertMinkowski}
Ben Lambert.
\newblock The perpendicular {N}eumann problem for mean curvature flow with a
  timelike cone boundary condition.
\newblock {\em Transactions of the American Mathematical Society}, (to appear).

\bibitem{Lieberman}
Gary~M. Lieberman.
\newblock {\em Second Order Parabolic Differential Equations}.
\newblock World Scientific Publishing Co. Pte. Ltd., 1996.

\bibitem{Stahlsecond}
Axel Stahl.
\newblock Convergence of solutions to the mean curvature flow with a {N}eumann
  boundary condition.
\newblock {\em Calculus of Variations and Partial Differential Equations},
  4:421--441, 1996.

\bibitem{Stahlfirst}
Axel Stahl.
\newblock Regularity estimates for solutions to the mean curvature flow with a
  {N}eumann boundary condition.
\newblock {\em Calculus of Variations and Partial Differential Equations},
  4:385--407, 1996.

\bibitem{Stampacchia}
Guido Stampacchia.
\newblock {\em Equations elliptiques au second ordre \`a co\'efficients
  discontinues}.
\newblock S\'eminaire de math\'ematiques sup\'erieures, 16. Les Presses de
  l'Universit\'e de Montreal, Montreal, 1966.

\bibitem{GigaOhnumaSatoNeumann}
Masaki~Ohnuma Yoshikazu~Giga and Moto-Hiko Sato.
\newblock On the strong maximum principle and the large time behaviour of
  generalised mean curvature flow with the {N}eumann boundary condition.
\newblock {\em Journal of Differential Equations}, 154:107--131, 1999.

\end{thebibliography}

\end{document}